\documentclass{ws-jnmp}

\usepackage[T1]{fontenc}
\usepackage{graphicx}

\usepackage{color}
\definecolor{MyLinkColor}{rgb}{0,0,0.4}

\newcommand{\R}{{\mathbb R}}

\newcommand{\Z}{{\mathbb Z}}

\newcommand{\p}{\partial}
\newcommand{\0}{\Omega}

\begin{document}

\newtheorem{thm}{Theorem}[section]
\newtheorem{prop}[thm]{Proposition}
\newtheorem{cor}[thm]{Corollary}

\arttype{Article} 

\markboth{A.--V. Matioc}
{An explicit solution for deep water waves with Coriolis effects}

%
%
\copyrightauthor{D. Henry}

\title{{\sc An explicit solution for deep water waves with Coriolis effects}}

\author{\footnotesize ANCA--VOICHITA MATIOC}

\address{Faculty of Mathematics, University of Vienna, \\ Nordbergstra\ss e 15, 1090 Vienna, Austria\\
\email{anca.matioc@univie.ac.at}}

\maketitle


\begin{abstract}
We present an explicit solution for the geophysical equatorial deep water waves in the {\em f}--plane approximation. 
\end{abstract}


\ccode{2010 Mathematics Subject Classification: Primary: 76B15; Secondary: 74G05, 37N10.}

\section{Introduction}
 When analyzing the motion of an object situated  in a  reference frame rotating with uniform angular velocity one has to take into consideration the Coriolis forces acting upon the object.
Particularly, when studying  water waves traveling over an inviscid  fluid at the Earth's surface, due to the rotation of the Earth around its axis, there are
 Coriolis and centrifugal forces which  appear and influence the motion of the fluid particles, cf. \cite{GS , Pe}. 
These forces add additional terms in the Euler equations and makes their analysis even more involved.

When neglecting the effects caused by the Earth's rotation, there exists an explicit solution for gravity deep water waves which was found first by Gerstner \cite{GE} and later on by  Rankine \cite{Ra}. 
Its features have been analyzed more recently in \cite{C, HD2}.
Gerstner's solution  describes the evolution of each individual fluid particle in the flow: each fluid particle moves  on a circle, and the  radius of the circle decreases with depth.  
Moreover, the flow is rotational and the vorticity decays with depth.

That for Gerstner's wave  fluid particles move on circles is in agreement with the  
 classical description of the
particle paths  within the framework of linear water wave theory \cite{ Cr,  RJ, LA, LI,  St}: all water particles
trace a circular orbit, the diameter of which decreases with depth.
However, it was recently shown within linear theory \cite{CEV, AMat1} that for irrotational periodic water waves the particle paths are not closed. 
Even within the linear water wave theory, the ordinary
differential equations system describing the motion of the fluid particles is nevertheless nonlinear and explicit solutions of
this system are not available.
However, qualitative features of the underlying flow have been obtained in a nonlinear setting in \cite{ConBook, CE2,   CEV,  DHH, AMat1, Mati, TO}.
The particle trajectories and other properties for the flow beneath water waves of finite depth have been obtained  in \cite{C2, CS, CV, HD1, IK, AMat}, to mention just some of the contributions.
  
It is worth mentioning that there exists an explicit solution describing the propagation of edge-waves along a sloping beach which is obtain in \cite{C1} by adapting Gerstner's solution. 
More recently, it was shown in \cite{ST} that Gerstner's idea may be used to construct an explicit solution for non-homogeneous deep water waves and for edge-waves  along a sloping beach.     

In this paper we consider the problem for  gravity deep water waves  in a  reference frame rotating with uniform angular velocity.
Neglecting the influence of centrifugal force, which is small compared to that due to the Coriolis force, and assuming that we are at the Equator, the mathematical formulation we deal with is the so-called $f-$plane approximation, see 
the discussion in \cite{C3} for the physical relevance of this approximation near the Equator.
For this problem we find an explicit  Gerstner's wave solution, which is a periodic wave over a rotational flow. 
Similarly as for waves without Coriolis effects, the particles move on circles and the radii of the circles, as well as the vorticity of the particles, decay with depth.

The outline of the paper is as follows:  after presenting in Section \ref{S:2} the governing equations for our problem we introduce in Section \ref{S:3}, by using a Lagrangian description, the 
explicit Gerstner solution  of the water wave problem in the $f-$plane approximation.
 
\section{The governing equations}\label{S:2}
In a rotating frame work with the origin at a point on earth's surface, with the $x-$axis chosen horizontally due east, the $y-$axis horizontally due north and the $z-$axis  upward, 
let $z=\eta(t,x,y)$ be the surface of the ocean. 
In the region $z\leq \eta(t,x,y)$ the governing equations in the $f-$plane approximation near the Equator are cf. \cite{GS} the Euler equations
\begin{equation}\label{Euler}
\left\{
\begin{array}{rllllll}
u_t+uu_x+vu_y+wu_z+2\omega w&=&-P_x/\rho, \\
v_t+uv_x+vv_y+wv_z&=& -P_y/\rho,\\
w_t+uw_x+vw_y+ww_z-2\omega u&=&-P_z/\rho -g,
\end{array}
\right.
\end{equation}
and, under the assumption of constant density, the equation of mass conservation in the form 
\begin{equation}\label{maco}
u_x+v_y+w_z=0.
\end{equation}
Here $t$ represents time, $(u,v,w)$ is the fluid velocity, $\omega=73\cdot 10^{-6} rad/s$ is the (constant) rotational speed of the Earth\footnote[1]{Taken to be a perfect sphere of radius $6371 km$.} round the polar axis
towards east, $\rho$ is the (constant) density of the water, $g=9,8 m/s^2$ is the (constant) gravitational acceleration  at the Earth's surface, and $P$ is the pressure.  
The free surface decouples the motion of the water from that of the air (see the discussion in \cite{ConBook}),
a fact that is expressed  by the dynamic boundary condition
\begin{equation}\label{Dbc}
P=P_0 \quad \quad \text{on} \quad z=\eta(t,x,y),
\end{equation}
where $P_0$ is the (constant) atmospheric pressure.

In this paper we seek two-dimensional flows, independent upon the $y-$coordinate and with $v \equiv 0$ throughout the flow.
Such flows are possible in this setting; in particular, the vorticity equation (see \cite{Pe}) ensures that the vorticity $\gamma=(0, u_z-w_x, 0)$ is preserved.
By abuse of notation we identify $\gamma$ with the scalar $u_z-w_x.$ 

Since the same particles always form the free surface, we also have the kinematic
boundary condition
\begin{equation}\label{kbc}
w=\eta_t+u\eta_x \quad \quad \text{on} \quad z=\eta(t,x),
\end{equation}
expressing the fact that a particle on the free surface remains confined to it (see the discussion in \cite{ConBook}. 

The boundary condition at the bottom
\begin{equation}\label{boc}
(u,w) \to (0,0) \, \, \text{as} \, z\to -\infty \, \, \text{uniformly for}\, \, x\in \R, \, t\geq 0,
\end{equation}
expresses  the fact that at great depths there is practically no motion. Summarizing, the governing equations for geophysical deep-water
waves  in the $f-$plane approximation are encompassed by the nonlinear free-boundary problem
\begin{equation}\label{fbpb}
\left\{
\begin{array}{rllllllll}
u_t+uu_x+wu_z+2\omega w&=&-P_x/\rho &\quad \text{for} \quad z<\eta(t,x),\\
w_t+uw_x+ww_z-2\omega u&=&-P_z/\rho -g  &\quad \text{for} \quad z<\eta(t,x),\\
u_x+w_z &=& 0  &\quad \text{for} \quad z<\eta(t,x),\\
u_z-w_x&=&\gamma(x,z) &\quad \text{for} \quad z<\eta(t,x),\\
P&=&P_0 &\quad  \text{on} \quad z=\eta(t,x),\\
w&=& \eta_t+u\eta_x &\quad  \text{on} \quad z=\eta(t,x),\\
(u,w) &\to& (0,0)  &\quad  \text{as} \quad z \to -\infty.
\end{array}
\right.
\end{equation}

\section{Gerstner's wave}\label{S:3}
We  prove herein that 
the problem \eqref{fbpb} has an explicit Gerstner-like solution 
solution.
Similarly as for waves without Coriolis effects, in this setting Gerstner's wave is a two-dimensional wave which is defined in Lagrangian framework by describing the evolution of individual water particles.  
The water particles are parametrized  by using two parameters $a\in \R$ and $b\leq b_0$ for a fixed $b_0\leq0.$
 Choosing a particular particle, that is  fixing $a$ and $b,$ its evolution  is described by the following relations  
\begin{equation}\label{Gers}
 \left \{
\begin{array}{rllll}
 x(t,a,b)&=& a-\frac{e^{kb}}{k}\sin(ka-kct), \\[1ex]
 z(t,a,b)&=& b+\frac{e^{kb}}{k}\cos(ka-kct), 
\end{array}
\right.
\end{equation}
where $k>0$ is fixed and $c>0$ has to be determined. The path of this particle is a circle centered at $(a,b)$ with radius $e^{kb}/k,$ with the particle moving clockwise with constant angular speed $kc.$
We can obtain the description of motion of another particle by changing the values of $a$ and $b$ in \eqref{Gers}.

In order to show that \eqref{Gers}
 provides a solution of the the problem \eqref{fbpb} we  prove that:
\begin{enumerate}
 \item there exist a suitable constant $c$ and a pressure function $P$ such that Euler's equations (the first two equations in \eqref{fbpb})  and  the fifth equation in \eqref{fbpb} are satisfied;
 \item the equation of mass conservation (the third equation in \eqref{fbpb})  holds;
 \item a particle on the free surface remains there (the sixth equation in \eqref{fbpb});
 \item the limiting boundary condition \eqref{boc} is satisfied; 
 \item  the vorticity of the flow, i.e. the function  $\gamma,$ is non-zero and decays with depth.
\end{enumerate}

 \subsection{Lagrangian coordinates}
The Lagrangian description of motion is the appropriate one for describing Gerstner's wave. 
In Lagrangian coordinates the variables $x$ and $z,$ denoting the position (in the physical frame) of a particle at time $t,$ are functions of $(a,b).$
Set $\Sigma_0:=\R\times(-\infty,b_0)$ for some $b_0\leq0.$
For each $t\geq0$, we define the mapping 
\begin{align*}
 \Phi(t)(a,b):=(x(t,a,b), z(t,a,b)) \quad \text{for all $  (a,b)\in \Sigma_0,$}
\end{align*}
where $(x(t,a,b), z(t,a,b))$ are given by \eqref{Gers}. The Jacobi matrix of this transformation is then given by
\begin{align}\label{pfi}
 \partial \Phi(t) = 
 \left( \begin{array}{ccc}
x_a & x_b \\
z_a & z_b 
 \end{array} \right)
=
\left( \begin{array}{ccc}
1-e^{kb}\cos(ka-kct) & -e^{kb}\sin(ka-kct)\\
-e^{kb}\sin(ka-kct) &1+e^{kb}\cos(ka-kct) \end{array} \right).
\end{align}
As first result, we state the following lemma, for the proof of which we refer to \cite{ConBook}.
\begin{lemma}\label{dife}
 Given  $t\geq 0,$ the function $\Phi(t)$ defines a real-analytic diffeomorphism from $\Sigma_0$ onto  its image $\0(t).$
Moreover, there exists a function $\eta:\R\to\R,$ which is periodic of period $2\pi/k$ such that 
\[
\0(t)=\{(x,z)\,:\, \text{$x\in\R$  and $z<\eta(t,x):=\eta(x-ct)$}\}.
\] 
\end{lemma}
\begin{remark}\label{R:1}
Since the map $a\mapsto x(t,a,b)$ is a bijection from $\R$ to $\R$ its inverse $x^{-1}(t, \cdot, b_0):\R \to \R$ is well defined.
Then, the function $\eta$ is defined by  $\eta(\zeta):=z(0,x^{-1}(0,\zeta,b_0),b_0).$

When $b_0=0,$ then the graph of the function $\eta$, the wave profile, is a cycloid and is real analytic excepting the points where $x=\pi m/k$, $m\in\Z$, while for $b_0<0$
is a real-analytic curve, called trochoid, see \cite{C}. 
This is due to the fact that the wave profile at time $t$, $t\geq0,$ is the image of $\R\times\{b_0\}$ under $\Phi(t).$  
\end{remark}
For further computations, it is useful to determine the Jacobian matrix corresponding to the inverse of $\Phi$
\begin{equation}\label{pfimin}
\begin{array}{rllll}
 \partial \left(\Phi^{-1}(t)\right) &=&
 \left( \begin{array}{ccc}
a_x & a_z \\
b_x & b_z 
 \end{array} \right)\\[2ex]
&=&
\displaystyle \frac{1}{1-e^{2kb}}\left( \begin{array}{ccc}
1+e^{kb}\cos(ka-kct) & e^{kb}\sin(ka-kct) \\
e^{kb}\sin(ka-kct) &1-e^{kb}\cos(ka-kct) \end{array} \right).  
\end{array}
\end{equation}
Setting $\Phi^{-1}(t)(x,z):=(a(t,x,z), b(t,x,z))$ for $ (x, z)\in \Omega(t)$, we have
\begin{equation*}
 \left \{
\begin{array}{rlllll}
 a(t, x(t, a,b), z(t;a,b))&=&a, \\
 b(t, x(t, a,b), z(t;a,b))&=&b. 
\end{array}
 \right.
\end{equation*}
Differentiating now in the last expression with respect to $t$ we obtain 
\begin{equation}\label{abt}
 \left \{
\begin{array}{rlllll}
 a_t+a_xx_t+a_zz_t&=&0, \\
 b_t+b_xx_t+b_zz_t&=&0. 
\end{array}
 \right.
\end{equation}
We look now for a constant $c$ such that the equations of the system \eqref{fbpb} are satisfied in the domains $\0(t)$ for all $t\geq0.$
Since the trajectories of the particle determined by $(a,b)\in\Sigma_0$ is given by \eqref{Gers}, the components of the velocity vector are found by differentiating these relations with respect to time 
\begin{equation}\label{Lag}
 \left \{
\begin{array}{rlllll}
 u(t,x,z)&=&x_t(t,a(t,x,z),b(t,x,z)), \\
 w(t,x,z)&=&z_t(t,a(t, x,z),b(t,x,z)). 
\end{array}
 \right.
\end{equation}
Differentiating in \eqref{Lag} with respect to $t, x$ and $z$, respectively, we get the following derivatives of $u$ and $w$
 \begin{equation}\label{u+w}
\left\{  
\begin{array}{rlllll}
   u_t&=& x_{tt}+x_{ta}a_t+x_{tb}b_t, \\
   u_x&=& x_{ta}a_x+x_{tb}b_x,\\
   u_z&=& x_{ta} a_z+x_{tb}b_z,
  \end{array}
\right.
\qquad\text{and}\qquad
\left\{  
\begin{array}{rlllll}
   w_t&=&z_{tt}+z_{ta}a_t+z_{tb}b_t, \\
   w_x&=& z_{ta}a_x+z_{tb}b_x,\\
   w_z&=& z_{ta} a_z+z_{tb}b_z.
  \end{array}
\right.
\end{equation}
\begin{lemma}\label{tt}
 Given $t\geq0,$ we have that 
 \begin{equation}\label{x+ztt}
  u_t+uu_x+wu_z=x_{tt}
\qquad\text{and}\qquad
  w_t+uw_x+ww_z=z_{tt}\qquad\text{in $\0(t).$}
 \end{equation}
\end{lemma}
\begin{proof}
 Using the relation \eqref{u+w}, we get
\begin{align*}
 u_t+uu_x+wu_z=&x_{tt}+x_{ta}a_t+x_{tb}b_t+x_tx_{ta}a_x+x_tx_{tb}b_x+z_tz_{ta}a_z+z_tx_{tb}b_z.
\end{align*}
Furthermore, if we multiply the first relation in \eqref{abt} with $x_{ta}$ and the second one by $x_{tb}$ we find
\begin{align*}
 x_{ta}a_t=-x_tx_{ta}a_x-z_tx_{ta}a_z\qquad\text{and}\qquad x_{tb}b_t=-x_tx_{tb}b_x-z_tx_{tb}b_z.
\end{align*}
Plugging these expressions into the previous relation, we obtain the first identity in \eqref{x+ztt}.
 The second identity is obtained in a  similar manner. 
\end{proof}
In order to verify Euler's equations, we have to determine  also a  function $P=P(t,x,z)$ for the pressure.
By the change of variables between the Lagrangian and Eulerian  coordinates we have $P_x=P_a a_x+P_b b_x$ and $P_z=P_a a_z+P_b b_z,$  
so that  the first two equations in \eqref{fbpb} are equivalent with
\begin{equation}\label{eul}
 \left\{
\begin{array}{rllll}
 \rho x_{tt}+2 \omega \rho z_t&=& -P_aa_x-P_b b_x,\\
 \rho z_{tt}-2 \omega \rho x_t&=& -P_a a_z-P_b b_z-g\rho.
\end{array}
\right.
\end{equation}
\begin{lemma}\label{press}
The speed of the wave is 
\begin{align}\label{c}
 c:=\frac{\sqrt{\omega^2+kg}-\omega}{k}
\end{align}
 and the pressure is given by the following expression
\begin{align}\label{presu}
 P=P_0+\frac{\rho(kc^2+2\omega c)}{2k}\left(e^{2kb}-e^{2kb_0}\right)-\rho g (b-b_0). 
\end{align}
\end{lemma}
\begin{proof}
Writing the system \eqref{eul} in matrix form we have
\begin{align*}
 \left(P_a \quad P_b \right) \left( \begin{array}{ccc}
a_x & a_z \\
 b_x & b_z 
 \end{array} \right) = \left(-\rho x_{tt}-2\omega \rho z_t  \quad -\rho z_{tt}+2\omega \rho x_t-g\rho  \right),
\end{align*}
and, using the relation $[\p\Phi^{-1}(t)]^{-1}=\p\Phi(t)$, we find the following equivalent formulation of \eqref{eul}: 
\begin{align}
 P_a&= -(\rho x_{tt}+2 \omega \rho z_t)x_a-(\rho z_{tt}-2 \omega \rho x_t+g\rho) z_a, \label{presab1}\\
 P_b&= -(\rho x_{tt}+2 \omega \rho z_t)x_b-(\rho z_{tt}-2 \omega \rho x_t+g\rho) z_b. \label{presab2}
\end{align}
Taking into account that 
\begin{equation}\label{timeder}
\begin{array}{rllll}
&x_t=ce^{kb}\cos(ka-kct), \qquad  &x_{tt}=kc^2e^{kb}\sin(ka-kct),\\
&z_t=ce^{kb}\sin(ka-kct), \qquad  &z_{tt}=-kc^2e^{kb}\cos(ka-kct),
\end{array} 
 \end{equation}
and using \eqref{pfi}, the relation \eqref{presab1} may be written the following form
\begin{align*}
 P_a=& [e^{kb}\cos(ka-kct)][( k\rho^2 e^{kb}\sin(ka-kct)+2\omega \rho c e^{kb}\sin(ka-kct)] \\
&-e^{kb}\sin(ka-kct)[gkc^2e^{kb}\cos(ka-kct)+2\omega \rho c e^{kb}\cos(ka-kct)-g\rho]\\
=&\rho e^{kb}\sin(ka-kct)[g-kc^2-2\omega c].
\end{align*}
Note from the fifth relation in \eqref{fbpb} that   pressure does not depending on the variable $x$ on the wave surface $z=\eta(x-ct)$.
Therefore, cf. Lemma \ref{dife}, it should be independent of $a$, at least when $b=b_0.$
This argument  leads us to the following restriction 
\begin{align}\label{quac}
 kc^2+2\omega c-g=0,
\end{align}
which is a quadratic equation in $c.$
 Solving this equation and taking into account that the velocity $c$ has to be positive we obtain that $c$ is given by relation \eqref{c}, and therewith  $P_a=0$ in $\0(t)$ for all $t\geq0.$ 
 Using  \eqref{quac} and \eqref{pfi},  relation \eqref{presab2} may be transformed as follows
\begin{align*}
 P_b=&e^{kb}\sin(ka-kct) [\rho kc^2 e^{kb}\sin(ka-kct)+2 \omega \rho c e^{kb} \sin(ka-kct)] \\
&+[1+e^{kb}  \cos(ka-kct)] [\rho k c^2 e^{kb}\cos(ka-kct)+2\rho\omega c e^{kb}  \cos(ka-kct)-\rho g]\\
=&\rho e^{2kb}(kc^2+2\omega c)-\rho g.
\end{align*}
 We integrate now the  last expression with respect to  $b$, and since $P=P_0$ on the wave surface $b=b_0,$ we find the desired expression for the pressure.
\end{proof}
In order to prove that the fluid is incompressible we need the following relations
\begin{equation}\label{mixed}
\begin{array}{rllll}
 &x_{at}=-kce^{kb}\sin(ka-kct), \qquad &x_{at}=kce^{kb}\cos(ka-kct),\\
 &x_{bt}=kce^{kb}\cos(ka-kct),\qquad &x_{bt}=kce^{kb}\sin(ka-kct).
\end{array} 
\end{equation}
\begin{lemma}\label{incomp}
The fluid is incompressible, that is for all $t\geq0$ we have
\[
 u_x+w_z=0\qquad\text{in $\0(t)$}.
\] 
\end{lemma}
\begin{proof}
Using the relations \eqref{pfimin} and \eqref{mixed} we find 
\begin{align*}
 u_x+w_z=& x_{ta}a_x+x_{tb}b_x+z_{ta}a_z+z_{tb}b_z=0.
\end{align*}
\end{proof}
In the next lemma we will show  that the flow is rotational and that the vorticity decays with depth.
\begin{lemma}\label{vorti}
 Given $(a,b)\in\Sigma_0$, the vorticity of the water particle determined by this coordinates  depends only on the parameter $b$ and is given by
\begin{align}\label{vort}
 \gamma=-\frac{2kce^{2kb}}{1-e^{2kb}}.
\end{align}
\end{lemma}
\begin{proof}
 Gathering \eqref{pfimin}, \eqref{u+w}, and \eqref{mixed}, it follows by direct computations that
\begin{align*}
 \gamma=&u_z-w_x=x_{ta}a_z+x_{tb}b_z-z_{ta}a_x-z_{tb}b_x=-\frac{2kce^{2kb}}{1-e^{2kb}}.
\end{align*}
We note that the vorticity function is negative and is strictly decreasing as a function of $b$. 
Particularly, when $b\to-\infty,$ that is at big depths, the vorticity of the flow decreases to zero.
\end{proof}
We verify next the kinematic boundary condition, i.e. the sixth equation of system \eqref{fbpb}.
To this end, we recall  that  the fluid's surface $z=\eta(t,x)=\eta(x-ct) $ consists on the same particles $(x(t,a,b_0),y(t,a,b_0))$, cf. Remark \ref{R:1}.   
\begin{lemma}\label{free}
Given $t\geq0,$ we have that
\begin{align}\label{dybc}
w=\eta_t+u\eta_x \quad \text{on $ z=\eta(t,x)$}.
\end{align}
\end{lemma}
\begin{proof}
In order to show that the relation \eqref{dybc} is satisfied we observe that \eqref{dybc} is equivalent to the relation $w=(u-c)\eta_x$.
By the definition of $\eta$ we have:
\[
\eta_x(x(t,a,b_0))=\frac{z_a}{x_a}(t,a,b_0)=-\frac{e^{kb_0}\sin(ka-kct)}{1-e^{kb_0}\cos(ka-kct)},
\]
cf. \eqref{pfi}.
Using \eqref{Lag} and the previous relation we obtain that
\begin{align*}
 w -(u -c)\eta_x =&ce^{kb_0}\sin(ka-kct)\\
&+c(e^{kb_0}\cos(ka-kct)-1)\frac{e^{kb_0}\sin(ka-kct)}{1-e^{kb_0}\cos(ka-kct)}\\
=&0,
\end{align*}
 and therefore the kinematic surface condition is satisfied.
\end{proof}
Regarding the limiting boundary condition, we note that
\begin{align*}
 (u,w)(t, x, z)=(ce^{kb}\cos(ka-kct), ce^{kb}\sin(ka-kct)), \qquad (x,z)\in\0(t),
\end{align*}
whereby $(a,b)=(a,b)(t,x,z).$
Since $b(t,x,z)\searrow-\infty$ when $z\searrow-\infty,$ we deduce that $(u,w)(t,x,w)\to0$ as $z\to-\infty.$
This shows that Gerstner's wave give by the Lagrangian description \eqref{Gers} is indeed an explicit solution of the $f-$plane approximation \eqref{fbpb}.

\vspace{0.7cm}
\hspace{-0.5cm}{\large \bf Acknowledgement}\\[2ex]
This research has been supported by the FWF Project I544 --N13 ``Lagrangian kinematics of water waves'' of the Austrian Science Fund.\\

\end{document}